\documentclass[twoside]{article}
\usepackage[a4paper]{geometry}
\usepackage[latin1]{inputenc} 
\usepackage[T1]{fontenc} 
\usepackage{RR}
\usepackage{hyperref}

\usepackage{amssymb}
\usepackage{amsmath}

\usepackage{amsthm}
\usepackage{graphicx}
\usepackage{cite}
\usepackage{comment}

\newtheorem{prop}{Proposition}
\newtheorem{lemma}{Lemma}

\newtheorem{definition}{Definition}
\newtheorem{main}{Main Convergence Result}
\newtheorem{col}{Corollary} 

\newtheorem{ass}{Assumption}

\RRNo{8501}
\RRdate{March 2014}
 \RRversion{2}
 \RRdater{September 2014}
\RRauthor{
Mahmoud El Chamie\thanks{INRIA Sophia Antipolis-M{\'e}diterran{\'e}e, 2004 route des Lucioles - BP 93. 06902 Sophia Antipolis Cedex, France. Email:   mahmoud.el\_chamie@inria.fr. } \and
 
Ji Liu 
 \thanks[sfn]{Coordinated Science Laboratory, University of Illinois at Urbana-Champaign, Urbana, IL 61801, USA. Emails: \{jiliu,basar1\}@illinois.edu. }  \and

Tamer Ba\c{s}ar
\thanksref{sfn}}
\authorhead{M. El Chamie \& J. Liu \& T. Ba\c{s}ar}
\RRtitle{Conception et Analyse d'Algorithmes Distribu{\'e}s de Moyennage avec Valeurs {\'E}chang{\'e}es Discr{\'e}tis{\'e}es}
\RRetitle{Design and Analysis of Distributed Averaging with Quantized Communication}
\titlehead{Design and Analysis of Distributed Averaging with Quantized Communication}
\RRresume{Nous allons nous int{\'e}resser {\`a} un r{\'e}seau dont les n{\oe}uds, ou agents,  ont des valeurs initiales.   Nous souhaitons concevoir un algorithme ayant pour objectif  la convergence vers une valeur qui est la plus proche possible de la moyenne de toutes les valeurs initiales des n{\oe}uds.  Cette algorithme est bas{\'e}e sur les interaction entre les n{\oe}uds, o{\`u} un n{\oe}ud interagit avec un autre n{\oe}ud si ils sont voisins dans le graphe. Un tel algorithme est commun{\'e}ment appel{\'e} ``moyenne distribu{\'e}e''. L'objectif de cet article est d'{\'e}tudier les performances d'une sous-classe d'algorithmes d{\'e}terministes de calcul de la moyenne distribu{\'e}e, o{\`u} l'{\'e}change d'informations entre les n{\oe}uds voisins est soumis {\`a} la quantification uniforme. Avec une telle quantification, la moyenne pr{\'e}cise ne peut {\^e}tre atteinte (sauf dans des cas exceptionnels), mais une valeur proche d'elle peut {\^e}tre atteinte. Cette valeur est appel{\'e}e consensus quantifi{\'e}. Nous montrons dans ce papier que, dans un temps fini, soit tous les  agents parviennent {\`a} un consensus quantifi{\'e} o{\`u} la valeur de consensus est le plus grand entier qui n'est pas sup{\'e}rieur {\`a} la moyenne de leurs valeurs initiales; ou soit tous les agents cyclent  dans un petit voisinage autour de la moyenne, en fonction des conditions initiales. Dans ce dernier cas, il est d{\'e}montr{\'e} que le voisinage peut {\^e}tre rendue arbitrairement faible en ajustant les param{\`e}tres de l'algorithme de mani{\`e}re distribu{\`e}e.
}
\RRabstract{
Consider a network whose nodes have some initial values, and it is desired to design an algorithm that builds on neighbor to neighbor interactions with the ultimate goal of convergence to the average of all initial node values or to some value close to that average. Such an algorithm is called generically ``distributed averaging'', and our goal in this paper is to study the performance of a subclass of deterministic distributed averaging algorithms where the information exchange between neighboring nodes (agents) is subject to uniform quantization. With such quantization, convergence to the precise average cannot be achieved in general,  but the convergence would be to some value close to it, called quantized consensus. Using Lyapunov stability analysis, we characterize the convergence properties of the resulting nonlinear quantized system. We show that in finite time and depending on initial conditions, the algorithm will either cause all agents to reach a quantized consensus where the consensus value is the largest quantized value not greater than the average of their initial values, or will lead all variables to cycle in a small neighborhood around the average. In the latter case, we identify tight bounds for the size of the neighborhood and we further show that the error can be made arbitrarily small by adjusting the algorithm's parameters in a distributed manner. }
\RRmotcle{distributed averaging, quantization, finite state automata, cycle, quantized consensus}
\RRkeyword{distributed averaging, quantization, finite state automata, cycle, quantized consensus}
\RRprojet{Maestro}
\RCSophia 

\begin{document}
\makeRR   
\tableofcontents

\newpage
\section{Introduction}

There has been considerable interest recently in developing algorithms for distributing information among  members  of interactive agents via local interactions (e.g.,  a group of sensors \cite{Avrachenkov2011} or mobile autonomous agents  \cite{lynch}), especially for the scenarios where agents or sensors are constrained by limited sensing, computation, and communication capabilities.
Notable among these  are those algorithms intended to cause such a group to reach
a consensus in a distributed manner \cite{Ts3,vicsekmodel,blondell}. 
Consensus processes play an important role in many other problems such as  
Google's PageRank \cite{pagerank}, clock synchronization \cite{Schenato:2007}, 
and formation control \cite{fax}.

One particular type of consensus process, distributed averaging, has received much attention lately 
\cite{Xiao04, moura1,El-Chamie:2012:8078, Morse2011}.
In its simplest form, distributed averaging deals with a network of $n>1$ agents and
the constraint that each agent $i$ is able to communicate only with  
certain other agents called agent $i$'s neighbors. Neighbor relations are
conveniently characterized  by a simple,   connected graph in which vertices correspond to agents
 and edges indicate neighbor relations.   Each agent $i$ initially
has  or  acquires  a real number $z_i$ which might be a measurement value.
 The  {\em distributed averaging problem}  is to devise
 an algorithm which will enable each agent to compute the average
$z_{ave} = \frac{1}{n}\sum_{i=1}^nz_i$ using only information acquired from  its neighbors.

Most existing algorithms for precise distributed averaging 
require that agents are able to send and receive real values with infinite precision. 
However, a realistic network can only allow messages with limited length to be transmitted between agents
due to constraints on the capacity of communication links.
With such a constraint, when a real value is sent from an agent to its neighbors, 
this value will be truncated and only a quantized version will be received by the neighbors. 
With such quantization, the precise average cannot be achieved
(except in particular cases), but some value close to it can be achieved, called quantized consensus.
A number of papers have studied this quantized consensus problem and
various \emph{probabilistic} strategies have been proposed to cause all the agents in a network 
to reach a quantized consensus with probability one (or at least with high probability)
\cite{Schuchman:64,Aysal:10,aysal,vote1,vote,murray,moura,Basar07, Etesami2013}. 
Notwithstanding this, the problem of how to design and analyze \emph{deterministic} algorithms for quantized consensus remains open \cite{Frasca:2009Ave,Cao2013}.

In this paper, we thoroughly analyze the performance of a deterministic distributed averaging algorithm
where the information exchange between neighboring agents is subject to uniform quantization. 
It is shown that in finite time, the algorithm will either cause all $n$ agents to reach 
a quantized consensus where the consensus value is the largest integer not greater than 
the average of their initial values, or will lead all $n$ agents' variables to cycle in a small neighborhood 
around the average, depending on initial conditions. 
In the latter case, it is further shown that the neighborhood can be arbitrarily small 
by adjusting the algorithm's parameters in a distributed manner.

The rest of the paper is organized as follows: in Section~\ref{ch4:sec:LR}  we review the existing literature related to our work. In Section~\ref{ch4:sec:DA} we introduce some preliminaries of distributed averaging. A network model for quantized communications is given in Section~\ref{ch4:sec:QC} . In Section~\ref{ch4:sec:PF}, we formulate the problem considered in this paper and present the equation model of the quantized system.  The design and analysis of the system, including the main results of the paper, are given in Section~\ref{ch4:sec:DAS}. A further discussion is given in Section~\ref{ch4:D}. Section~\ref{ch4:sec:S} provides some simulations supporting our analytic results and  Section~\ref{ch4:sec:C} concludes the paper.

\section{Literature Review}\label{ch4:sec:LR}

Most of the related works for distributed averaging with quantized communication propose either a deterministic algorithm (as our approach in this paper) or a probabilistic one.

There are only a few publications which study deterministic algorithms for quantized consensus.  In \cite{Tao:11} the distributed averaging problem with quantized communication
is formulated as a feedback control design problem
for coding/decoding schemes; the paper characterizes the amount of information needed to be sent
for the agents to reach a consensus and shows that with an appropriate
scaling function and some carefully chosen control gain,
the proposed protocol can solve the distributed averaging problem,
but some spectral properties of the Laplacian matrix of the underlying fixed undirected graph have to be known in advance.  
More sophisticated coding/decoding schemes were proposed in \cite{xieauto11} for time-varying undirected graphs
and in \cite{zhang13} for time-varying directed graphs,
all requiring carefully chosen parameters. Recently a novel dynamic quantizer has been proposed in \cite{Thanou2013} based on  dynamic quantization intervals for coding of the exchanged messages in wireless sensor networks leading to asymptotic convergence to consensus.
In \cite{murrayacc} a biologically inspired algorithm was proposed which will cause all $n$ agents to
reach some consensus with arbitrary precision, but at the cost of not preserving the desired average.
Control performance of logarithmic quantizers was studied in \cite{carliauto08} and
quantization effects were considered in \cite{Nedic:2009Ond}.
A deterministic algorithm of the same form as in this paper
has been only partially analyzed in \cite{Frasca:2009Ave} where the authors have approximated the system
by a probabilistic model and left the design of the weights as an open problem.

Over the past decade quite a few probabilistic quantized consensus algorithms have been proposed.
The probabilistic quantizer in \cite{aysal}
ensures almost surely consensus at a common but random quantization level for
fixed (strongly connected) directed graphs;
although the expectation of the consensus value equals the
desired average, the deviation of the consensus value from the desired average is not tightly bounded.
An alternative algorithm which gets around this limitation was proposed in \cite{moura}; 
the algorithm adds dither to the agents' variables before quantization
and the mean square error can be made arbitrarily small by tuning the parameters. 
The probabilistic algorithm in \cite{vote1,vote}, called ``interval consensus gossip'', causes all $n$ agents
to reach a consensus in finite time almost surely on the interval in which the average lies,
for time-varying (jointly connected) undirected graphs.
A stochastic quantized gossip algorithm was shown to work properly in \cite{murray}.
The effects of quantized communication on the standard randomized gossip algorithm \cite{boyd052}
were analyzed in \cite{Carli:2010Gos}.
An alternative approach to analyze the quantization effect was introduced in \cite{Schuchman:64,Aysal:10}
which model the effect as noise following certain probability.

Another thread of research has studied quantized consensus with the additional constraint
that the value at each node is an integer.
The probabilistic algorithm in \cite{Basar07}
causes all $n$ agents to reach quantized consensus almost surely
for a fixed (connected) undirected graph;
convergence time of the algorithm was studied in \cite{Etesami2013}, with strong bounds on its expected value.
In \cite{cai} a probabilistic algorithm was proposed to solve the quantized consensus problem
for fixed (strongly connected) directed graphs using the idea of ``surplus''.

We should note that, in addition, our work in this paper is also related to the literature on the problem of
load balancing \cite{rao,subramanian,ghosh1}.

\section{Distributed Averaging}\label{ch4:sec:DA}

Consider a group of $n > 1$ agents labeled $1$ to $n$. 
Each agent $i$ has control over a real-valued scalar quantity $x_i$ called an {\em agreement variable} 
which the agent is able to update its value from time to time.
Agents may only communicate with their ``neighbors''. 
Agent $j$ is a {\em neighbor} of agent $i$ if $(i,j)\in\mathcal{E}$ is an edge in a given 
simple, undirected $n$-vertex graph $\mathbb{G}=(\mathcal{V},\mathcal{E})$ 
where $\mathcal{V}=\{1,2,\ldots,n\}$ is the vertex set and 
$\mathcal{E}$ is the edge set. 
We assume that the graph $\mathbb{G}$ is connected and does not change over time. 
Initially each agent $i$ has a real number $x_i(0)$. Let $$x_{ave}(k)=\frac{1}{n}\sum _{i\in \mathcal{V}}x_{i}(k),$$
be the average of values of all agreement variables in the network, we will refer to $x_{ave}(0)$ simply as $x_{ave}$. 
The purpose of the distributed averaging
problem is to devise an algorithm which enables all $n$ agents to asymptotically determine 
in a decentralized manner, the average of the initial values of their scalar variables, 
i.e.,
$$\lim _{k\rightarrow \infty} x_i(k)=x_{ave}.$$

A well studied approach to the
problem is for each agent to use a linear iterative update rule of the form
\begin{equation}\label{ch4:stateEq}
x_i(k+1)=w_{ii}x_i(k)+\sum_{j\in \mathcal{N}_i}w_{ij}x_j(k) , \ \ \forall i\in \mathcal{V},
\end{equation}
where $k$ is a discrete time index, 
$\mathcal{N}_i$ is the set of neighbors of agent $i$ and the $w_{ij}$ are real-valued weights 
to be designed.
In \cite{boyd04} several methods are proposed
for choosing the weights $w_{ij}$ with the goal of obtaining algorithms with improved convergence rates. One
particular choice, which defines what has come to be known as the Metropolis algorithm, requires
only local information to define the $w_{ij}$ \cite{metro,Xiao05distributedaverage}.
The corresponding Metropolis weights are chosen as follows:
\begin{eqnarray*}
w_{ij} &=& \frac{1}{\text{max}\{d_i,d_j\}+1},  \ \ \forall (i,j)\in \mathcal{E},\\
w_{ii} &=& 1-\sum _{j\in \mathcal{N}_i}w_{ij}, \ \ \forall i\in \mathcal{V},
\end{eqnarray*}
where $d_i$ is the degree of agent $i$.

Eq.~\eqref{ch4:stateEq} can be written in a matrix form as 
$$\mathbf{x}(k+1)=W\mathbf{x}(k),$$
where $\mathbf{x}(k)$ is the state vector of agreement values whose $i$th element equals $x_i(k)$, 
and $W$ is the weight matrix whose $ij$th entry equals $w_{ij}$.
It should be clear that $w_{ij}>0$ if $(i,j)\in \mathcal{E}$ and $w_{ij}=0$ otherwise. 
A necessary and sufficient condition  for the convergence of Eq.~\eqref{ch4:stateEq} to the desired average 
for any initial values is that $W$ is a doubly stochastic matrix 
and all eigenvalues of $W$, with the exception of a single eigenvalue of value $1$, have magnitude strictly
less than unity \cite{Xiao04}.
It is easy to verify that the Metropolis weights satisfy this condition. 
Thus the Metropolis weights guarantee the desired convergence, i.e.,
$$\lim _{k\rightarrow \infty}\mathbf{x}(k)=x_{ave}\mathbf{1},$$
where $\mathbf{1}$ is the vector in $\mathbb{R}^n$ whose entries all equal one. 
It is worth noting that since $W$ is doubly stochastic, the summation of all $n$ values of agreement variables is kept constant, so is the average of the variables, namely
$$\mathbf{1}^T\mathbf{x}(k)=\mathbf{1}^T\mathbf{x}(0)=nx_{ave}, \ \forall k.$$

\section{Quantized Communication}\label{ch4:sec:QC}
In a network where links have constraints on the capacity and have limited bandwidth (e.g., digital communication networks), messages cannot have infinite length. However, the distributed averaging algorithm requires sending real (infinite precision) values through these communication links.   Therefore, with digital transmission,  the messages transmitted between neighboring agents will have to be  truncated. If the communication bandwidth was limited, the more the truncation of agents' values,  the higher would be the deviation of agent's value from the desired average consensus $x_{ave}$.

To model the effect of quantized communication, we assume that the links  perform a quantization effect on the values transmitted between agents. The network model is given by Fig.~\ref{ch4:NetModel}.
\begin{figure}
\begin{center}
\includegraphics[scale=0.45]{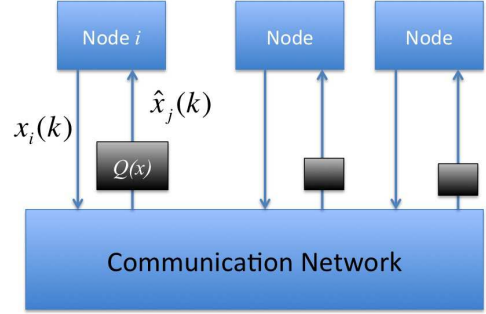}
\caption{The network model for the quantized system.  
} 
\label{ch4:NetModel}
\end{center}
\end{figure}
As we can see from the model, each agent $i$ can have infinite bandwidth to store its latest value $x_i(k)$ and perform computations. However, when agent $i$ sends its value at time $k$ through the communication network,  its neighbors will receive a value $\hat{x}_i(k)$ which is the quantized value of $x_i(k)$. A {\em quantizer} is a function $\mathcal{Q} : \mathbb{R}\rightarrow \mathbb{Z}$ that maps a real value to an integer. Quantizers can be of different forms. We present here  some widely used quantizers in the literature \cite{Nair:2007Fee, Carli:2010Gos, Nedic:2009Ond}:

\begin{enumerate}
\item Truncation quantizer $\mathcal{Q}_t$ which truncates the decimal part of a real number and keeps the integer part:
\begin{align}
\mathcal{Q}_t(x)&=\lfloor x\rfloor .
\end{align}

\item Ceiling quantizer $\mathcal{Q}_c$ which rounds the value to the nearest upper integer:
\begin{align}
\mathcal{Q}_c(x)&=\lceil x\rceil .
\end{align}

\item Rounding quantizer $\mathcal{Q}_r$ which rounds a real number to its nearest integer:
\begin{align}
\mathcal{Q}_r(x)&=
\begin{cases}
\lfloor x\rfloor &  \text{ if } x-\lfloor x\rfloor < 1/2  \\
\lceil x\rceil & \text{ if } x-\lfloor x\rfloor \geq 1/2.
\end{cases}
\end{align}

\item Probabilistic quantizer $\mathcal{Q}_p$ defined as follows: 
\begin{align}
\mathcal{Q}_p(x)=
\begin{cases}
\lfloor x\rfloor & \text{ with probability } \lceil x\rceil - x\\
\lceil x\rceil & \text{ with probability } x-\lfloor x\rfloor.
\end{cases}
\end{align}
\end{enumerate}
In this report we study the effect of the deterministic quantizers ($\mathcal{Q}_t(x)$, $\mathcal{Q}_c(x)$, and $\mathcal{Q}_r(x)$) on the performance of the distributed averaging algorithms by showing the distance that the agents' stored values can deviate from the initial average $x_{ave}$. The quantizers listed before map $\mathbb{R}$ into $\mathbb{Z}$ and have quantization jumps of size 1.  Quantizers having a generic real positive quantization step $\epsilon$ can be simply recovered by a suitable scaling: 
$\mathcal{Q}^{(\epsilon )}(x)=\epsilon \mathcal{Q}(x/\epsilon )$ \cite{Carli:2010Gos}. Thus the results in this report cover these generic quantizers as well.

\section{Problem Formulation}\label{ch4:sec:PF}

Suppose that all $n$ agents adhere to the same update rule of Eq.~\eqref{ch4:stateEq}. Then with a quantizer $\mathcal{Q}(x)$, the network equation would be
\begin{equation} \label{ch4:sys2}
x_i(k+1)=w_{ii} x_i(k) +\sum_{j\in \mathcal{N}_i}w_{ij}\mathcal{Q}(x_j(k)), \ \ \forall i\in \mathcal{V}.
\end{equation}
Simple examples show that this algorithm can cause the system to shift away from the initial average $x_{ave}$.

Since agents know exactly the effect of the quantizer, for the agents not to lose any information caused by quantization, at each iteration $k$ each agent $i$ can send out the quantized value $\mathcal{Q}(x_i(k))$ (instead of sending $x_i(k)$) and store in a \emph{local} scalar $c_i(k)$ the difference between the real value $x_i(k)$ and its quantized version, i.e.,
\begin{align*}
c_i(k)&=x_i(k)-\mathcal{Q}( x_i(k) ).
\end{align*}

Then, the next iteration update of agent $i$ can be modified to be
\begin{equation}
 x_i(k+1)=w_{ii} \mathcal{Q}(x_i(k)) +\sum_{j\in \mathcal{N}_i}w_{ij}\mathcal{Q}( x_j(k) ) +c_i(k), \ \ \forall i\in \mathcal{V}.
\end{equation}
A major difference between this equation and \eqref{ch4:sys2} is that here no information is lost; i.e., the total average is being conserved in the network, as we will show shortly after. The state equation of the system becomes,
\begin{equation}\label{ch4:quantizedstateeq}
\mathbf{x}(k+1)=W\mathcal{Q}\left(\mathbf{x}(k)\right) + \mathbf{x}(k) - \mathcal{Q}\left(\mathbf{x}(k)\right) ,
\end{equation}
where, with a little abuse of notation,  $\mathcal{Q}\left(\mathbf{x}\right)=\left(\mathcal{Q}(x_1), \mathcal{Q}(x_2), \dots ,\mathcal{Q}(x_n) \right)^T$ is the vector quantization operation. For any  $W$ where each column sums to $1$ ($\mathbf{1}^TW=\mathbf{1}^T$ where $\mathbf{1}$ is the vector of all ones), the  total sum of all $n$ agreement
variables does not change over time if agents followed the
protocol of  Eq.~\eqref{ch4:quantizedstateeq}:
\begin{align}
 \mathbf{1}^T\mathbf{x}(k+1)&= \mathbf{1}^T(W\mathcal{Q}\left(\mathbf{x}(k)\right)  + \mathbf{1}^T\mathbf{x}(k) - \mathbf{1}^T\lfloor \mathbf{x}(k)\rfloor \nonumber\\
 &=\mathbf{1}^T\mathcal{Q}\left(\mathbf{x}(k)\right)  + \mathbf{1}^T\mathbf{x}(k) - \mathbf{1}^T\mathcal{Q}\left(\mathbf{x}(k)\right)  \nonumber\\
 &=\mathbf{1}^T\mathbf{x}(k)\nonumber\\
 &=\mathbf{1}^T\mathbf{x}(0)\nonumber\\
 &=nx_{ave},\label{ch4:AveCons}
 \end{align}
 Thus the average is also  conserved ($x_{ave}(k)=x_{ave}, \; \; \forall k$).
Equation \eqref{ch4:quantizedstateeq} would be our model of distributed averaging with deterministic quantized communication where the quantizer can take the form of the truncation $\mathcal{Q}_t$, the ceiling $\mathcal{Q}_c$, or the rounding one $\mathcal{Q}_r$. It is worth noting that the three quantizers can be related by the following equations:
\begin{align}
\mathcal{Q}_r(x)&=\mathcal{Q}_t(x+1/2),\\
\mathcal{Q}_c(x)&=-\mathcal{Q}_t(-x).
\end{align}

Given a model with the ceiling quantizer $\mathcal{Q}_c$ in  \eqref{ch4:quantizedstateeq}, by taking 
 $\mathbf{y}(k)=-\mathbf{x}(k)$, 
 the system evolves as:
\begin{align*}
\mathbf{y}(k+1)&=\mathbf{y}(k)+W\mathcal{Q}_t(\mathbf{y}(k))-\mathcal{Q}_t(\mathbf{y}(k))\\
\mathbf{y}(0)&=-\mathbf{x}(0).
\end{align*}
Therefore, by analyzing the above system which has a truncation quantizer $\mathcal{Q}_t$, we can deduce the performance of $\mathbf{x}(k)$ that satisfies equation \eqref{ch4:quantizedstateeq} with a ceiling quantizer $\mathcal{Q}_c$ because they are related by a simple equation ($\mathbf{y}(k)=-\mathbf{x}(k)$).  

Similarly, given a model with the rounding quantizer $\mathcal{Q}_r$ in  \eqref{ch4:quantizedstateeq}, by taking 
 $\mathbf{y}(k)=\mathbf{x}(k)+\frac{1}{2}\mathbf{1}$, 
 the system evolves as:
\begin{align*}
\mathbf{y}(k+1)&=\mathbf{y}(k)+W\mathcal{Q}_t(\mathbf{y}(k))-\mathcal{Q}_t(\mathbf{y}(k))\\
\mathbf{y}(0)&=\mathbf{x}(0)+\frac{1}{2}\mathbf{1}.
\end{align*}
Therefore, by analyzing the above system which has a truncation quantizer $\mathcal{Q}_t$, we can deduce the performance of $\mathbf{x}(k)$ that satisfies equation \eqref{ch4:quantizedstateeq} with a rounding quantizer $\mathcal{Q}_r$ because they are related by a simple translation equation ($\mathbf{y}(k)=\mathbf{x}(k)+\frac{1}{2}\mathbf{1}$). Therefore the
effects of all these three quantizers are essentially the same.

With this nontrivial observation in mind, we focus on the analysis
of the truncation quantizer only in the rest of this report. The
results can then be easily extended to the case of the other two quantizers.

In the sequel we will fully characterize the  behavior of  system \eqref{ch4:quantizedstateeq} and its convergence properties. 
But first, we have the following definition:
\begin{definition}
A network of $n$ agents reaches quantized consensus if there is an iteration $k_0$ such that 
$$\mathcal{Q}(x_i(k))=\mathcal{Q}(x_j(k)), \;\; \forall i,j \in \mathcal{V},\ \forall k\geq k_0.$$
\end{definition}

\section{Design and Analysis of the System}\label{ch4:sec:DAS}
In this section, we carry out the analysis of the proposed quantized system equation. By considering the truncation quantizer $\mathcal{Q}_t$ in \eqref{ch4:quantizedstateeq}, the system equation becomes:
\begin{equation}\label{ch4:quantizedstateeq2}
\mathbf{x}(k+1)=W\lfloor \mathbf{x}(k)\rfloor + \mathbf{x}(k) - \lfloor \mathbf{x}(k)\rfloor.
\end{equation}
This can be written in a distributed way for every $i\in \mathcal{V}$ as follows:
\begin{align}
x_i(k+1)&=x_i(k)+\sum _{j\in \mathcal{N}_i}w_{ji}\left(\lfloor x_j(k)\rfloor-\lfloor x_i(k)\rfloor \right),\\
&=x_i(k)+\sum _{j\in \mathcal{N}_i}w_{ji}L_{ji}(k),
\end{align}
where 
$$L_{ji}(k)\triangleq \lfloor x_j(k)\rfloor - \lfloor x_i(k)\rfloor=-L_{ij}(k).$$
The non-linearity of the system due to quantization complicates the analysis, and traditional stability analysis of linear systems (such as ergodicity, products of stochastic matrices, etc.) cannot be applied here as the system might not even converge. As demonstrated in the following subsection.

\subsection{Cyclic Example}\label{subsec:CE}
The purpose of the following example is to show that for a ``bad'' weight matrix design, the quantized system can cycle very far from the average. Consider the  two-nodes example of Fig.~\ref{ch4:toynet},
\begin{figure}
\begin{center}
\includegraphics[scale=0.5]{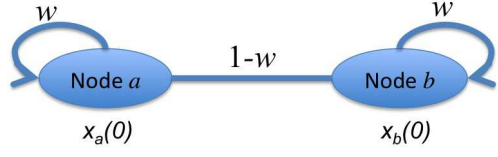}
\caption{Network of two nodes where quantized communication does not converge.  
} 
\label{ch4:toynet}
\end{center}
\end{figure}
suppose that $x_a(0)=\xi$, $x_b(0)=K +\xi$ where $K\in \mathbb{N}$ and $\xi \in (0,1)$. With these initial values, $\lfloor x_a(0)\rfloor =0$, $\lfloor x_b(0)\rfloor =K$, and $x_{ave}=\frac{K}{2}+\xi$. The weight matrix for this two-nodes system is assumed to be a doubly stochastic matrix and is given as follows:
$$W=\left( \begin{array}{cc}
w & 1-w  \\
1-w & w  \end{array} \right),$$
where $w\in (0,1)$. With this weight matrix, \eqref{ch4:AveCons} is satisfied and the average is conserved. In \cite{Frasca:2009Ave}, the authors defined the following metric to measure the performance of the system:
\begin{equation}
d_\infty (W, \mathbf{x}(0) )=\limsup _{k\rightarrow \infty}\frac{1}{\sqrt{n}}||\Delta (k)||,
\end{equation}
where $\Delta (k)$ is a vector having the elements $\Delta _i(k)=x_i(k)-x_{ave}$.
 So the worst cycle (according to this metric), given a doubly stochastic weight matrix, would happen if the nodes toggled their values with every iteration. Let us derive conditions on $W$ for which this could happen. With the quantization, the corresponding system equations are as follows:
\begin{align}
x_a(k+1)&=x_a(k)+(1-w)\times\left(\lfloor x_b(k)\rfloor - \lfloor x_a(k)\rfloor\right)\\
x_b(k+1)&=x_b(k)+(1-w)\times\left(\lfloor x_a(k)\rfloor - \lfloor x_b(k)\rfloor\right).
\end{align}
From the given initial conditions, after one iteration the updated values are $x_a(1)=\xi +(1-w)K$ and $x_b(1)=K+\xi -(1-w)K$. Therefore, the quantized value of the nodes' variables will toggle between $0$ and $K$ if $x_a(1)\in [K, K+1)$ and $x_b(1)\in [0,1)$. By substituting the values of $x_a(1)$ and $x_b(1)$ we get the following  conditions for such a cycle, 
\begin{equation}\label{ch4:conditionscycle}
\begin{cases}
wK>\max \{-\xi,\xi -1\}\\
wK<\min \{\xi, 1-\xi \}.
\end{cases}
\end{equation}
The first condition is always satisfied because $wK>0$. Then, a bad design of $W$ is to have $w<\frac{1}{K}\times \min \{\xi, 1-\xi \}$ because in this case the nodes can cycle\footnote{In case initial values were not known, since $\min \{\xi, 1-\xi \}\leq 1/2$, then, a bad design of $W$ is to have $w<\frac{1}{2K}$ because in this case there might be some initial values that cause large cycles.} with 
\begin{equation}
x_a(k)=
\begin{cases}
\xi &\text{ if $k$ is even}\\
K+\xi  -wK&\text{ if $k$ is odd}
\end{cases}
\ \  \text{and } \ \ 
x_b(k)=
\begin{cases}
K+\xi &\text{ if $k$ is even}\\
wK+\xi  &\text{ if $k$ is odd}.
\end{cases}
\end{equation}
Thus $\Delta _a(k)=\Delta _b(k)=K/2$ if $k$ is even, and so $d_\infty(W, \mathbf{x}(0))=K/2$. The above two-node network result can be extended to regular bipartite graphs where the first set of nodes takes the value $x_a(0)$ and the other set takes the value $x_b(0)$ and all self-weights are equal to $w$.\footnote{In case of hypercube graphs, \cite{Frasca:2009Ave} shows that if the weights in the network have a constant value $1/(d+1)$ where $d=\log n$ is the degree of a node in the hypercube graph, then an upper bound on $d_\infty (W)= \sup _{\mathbf{x}(0)} d_\infty (W, \mathbf{x}(0))$ is the following $d_\infty (W)\leq \frac{\log n}{2}$. Since a hypercube is a regular bipartite graph, then using our results leads to the following lower bound, $d_\infty (W)\geq \frac{\log n}{4}$ (by taking $\xi =0.5$ and $K=(\log n)/2$ to satisfy \eqref{ch4:conditionscycle}).} This would also lead to the following inequality on $d_\infty(W,\mathbf{x}(0))$ with the given initial conditions and weight matrix:
$$d_\infty (W,\mathbf{x}(0))\geq K/2.$$
This shows that a bad design of $W$ on general graphs can make the cycle arbitrarily large. 

\subsection{Weight Assumption}
The system behavior depends of course on the design of the weight matrix. In distributed averaging, it is important to consider weights that can be chosen locally, avoid \emph{bad} design, and guarantee desired convergence properties. We impose the following assumption on $W$ which can be satisfied in a distributed manner. 
 
\begin{ass}\label{ch4:ass1}
The weight matrix in our design has the following properties:
\begin{itemize}
\item $W$ is a symmetric doubly stochastic matrix:
\begin{align*}
w_{ij}=w_{ji}\geq 0 \ \ \forall i,j\in \mathcal{V}\\
\sum _iw_{ij}=\sum _jw_{ij}=1,
\end{align*}
\item Dominant diagonal entries of $W$: 
\begin{equation*}
w_{ii}>1/2 \text{ for all } i\in \mathcal{V},
\end{equation*}
\item Network communication constraint: if $(i,j)\notin \mathcal{E}$, then $w_{ij}=0$,
\item  For any link $(i,j)\in \mathcal{E}$ we have $w_{ij}\in \mathbb{Q}^+$, where $\mathbb{Q}^+$ is the set of rational numbers in the interval $(0,1)$.
\end{itemize}
\end{ass}
These are also sufficient conditions for the linear system \eqref{ch4:stateEq} to converge. The choice of weights being rational numbers is not restrictive because any practical implementation would satisfy this property intrinsically (we use it here to prove
convergence results). The dominant diagonal entries assumption is very important to prevent the system from having large cycles (as in the cyclic example in Section \ref{subsec:CE}).

We now state the main result of this report which will be proved in the following subsections.
\begin{main}\label{ch4:mainR}
Consider the quantized system \eqref{ch4:quantizedstateeq2}. Suppose that Assumption~\ref{ch4:ass1} holds. Then for any initial value $\mathbf{x}(0)$, there is a finite time iteration where either
\begin{enumerate}
\item  the system reaches quantized consensus, or
\item the nodes' values  cycle in a small neighborhood around the average, where the neighborhood can be made arbitrarily small by a decentralized design of the weights (having trade-off with the speed of convergence).
\end{enumerate} 
\end{main}
To highlight the importance of these results, notice that the Main Convergence Result~\ref{ch4:mainR} implies there is an iteration $k_0$ such that $x_i(k)-x_j(k)<1$ for all $i,j\in \mathcal{V}$ for $k\geq k_0$. This gives a constant upper bound  on the metric $d_\infty (W, \mathbf{x}(0))$ independent of initial values, i.e., due to Assumption~\ref{ch4:ass1}, $d_\infty (W, \mathbf{x}(0))\leq 0.5$ on any general graph and for any initial conditions. 
 
\subsection{Cyclic States}\label{ch4:subsec:CS}
We study in this subsection the convergence properties of the system equation \eqref{ch4:quantizedstateeq2} under Assumption~\ref{ch4:ass1}. Let us first show that  due to quantized communication, the states of the agents lie in a discrete set. Since $w_{ij}\in \mathbb{Q}^+$ for any link $(i,j)$, we can write $$w_{ij}=\frac{a_{ij}}{b_{ij}},$$ where $a_{ij}$ and $b_{ij}$ are co-prime  positive  integers. Suppose that $B_i$ is the Least Common Multiple (LCM) of the integers $\{b_{ij}; (i,j)\in \mathcal{E},j\in \mathcal{N}_i\}$. Let $c_i(k)= x_i(k)- \lfloor x_i(k)\rfloor$; then we have $c_i(k)\in [0,1)$. Let us see how $c_i(k)$ evolves:
\begin{align}
c_i(k)&=x_i(k)- \lfloor x_i(k)\rfloor \nonumber\\
&=x_i(k-1)+\sum _{j\in \mathcal{N}_i}w_{ij}\times \left(\lfloor x_j(k-1)\rfloor - \lfloor x_i(k-1)\rfloor\right)\nonumber\\
&\hspace*{0.5cm}- \lfloor x_i(k)\rfloor \nonumber\\
&=\lfloor x_i(k-1)\rfloor + c_i(k-1)\nonumber\\
&\hspace*{0.5cm}+\sum _{j\in \mathcal{N}_i}\frac{a_{ij}}{b_{ij}}\times \left(\lfloor x_j(k-1)\rfloor - \lfloor x_i(k-1)\rfloor\right)- \lfloor x_i(k)\rfloor \nonumber\\
&=c_i(k-1) +\frac{Z(k)}{B_i},
\end{align}
 where $Z(k)\in \mathbb{Z}$ is an integer.  Then with a simple recursion, we can see that for any iteration $k$ we have:
 \begin{equation}\label{ch4:StepLevels}
 c_i(k)=c_i(0)+\frac{\tilde{Z}(k)}{B_i},
 \end{equation}
 where $\tilde{Z}(k)\in \mathbb{Z}$. Since $c_i(k)\in [0,1)$, this equation shows that the states of the nodes are  quantized, and the decimal part can have maximum $B_i$ quantization levels. 

We now give the following definition,
\begin{definition}
The  quantized system \eqref{ch4:quantizedstateeq2}	is cyclic if there exists a positive integer $P$ and a finite time $k_0$ such that
 $$\mathbf{x}(k+P)= \mathbf{x}(k) \; \; \; \forall k\geq k_0,$$ 
 where $P$ is the cycle period.
\end{definition}

\begin{prop}\label{ch4:PropCycle}
Suppose Assumption~\ref{ch4:ass1} holds. Then, the quantized system \eqref{ch4:quantizedstateeq2}, starting from any initial value $\mathbf{x}(0)$,  is cyclic.  
\end{prop}
\begin{proof}
Let $m(k)$ and $M(k)$ be defined as follows:
\begin{equation}\label{ch4:minmax}
m(k)\triangleq \min_{i\in \mathcal{V}} \lfloor x_i(k)\rfloor,\;\; M(k)\triangleq \max_{i\in \mathcal{V}}\lfloor x_i(k)\rfloor.
\end{equation}
  Notice that for any $k$, we have
\begin{align*}
x_i(k+1)&=x_i(k)+\sum _{j\in \mathcal{N}_i}w_{ji}L_{ji}\\
&\leq c_i(k)+\lfloor x_i(k)\rfloor +\left(\sum _{j\in \mathcal{N}_i}w_{ji}\right)\left(M(k)-\lfloor x_i(k)\rfloor\right)\\
&\leq c_i(k)+M(k),
\end{align*}
 from which it follows that $\lfloor x_i(k+1)\rfloor\leq M(k)$, and hence $M(k+1)\leq M(k)$. By a simple recursion we can see that the maximum cannot increase, $M(k)\leq M(0)$. Similarly, we have $m(k)\geq m(0)$. As a result, $\lfloor x_i(k)\rfloor\in \{m(0), m(0)+1, \dots ,M(0)-1, M(0)\}$ is a finite set. Moreover, from equation  \eqref{ch4:StepLevels}, $c_i(k)$ belongs to a finite set that can have at most $B_i$ elements. Since $x_i(k)= \lfloor x_i(k)\rfloor + c_i(k)$, and each of the elements in the sum belongs to a finite set,  $x_i(k)$ belongs to a finite set as well. But from equation \eqref{ch4:quantizedstateeq2}, we have $\mathbf{x}(k+1)=f\left(\mathbf{x}(k)\right)$ where the function $f(.)$ is a deterministic function of the input state at iteration $k$, so the system is a deterministic finite state automata. States of deterministic automata enter a cycle in finite time \cite{Reger:2002Cyc}, and therefore the system is cyclic.  
\end{proof}

\subsection{Lyapunov Stability}

In this subsection, we will study the stability of the above system using a Lyapunov function. 
Assumption~\ref{ch4:ass1} and Eq.~\eqref{ch4:StepLevels} imply that there exists a fixed\footnote{By `fixed' we mean that the value is independent of time and it only depends on initial values and the network structure.} strictly positive constant $\gamma >0$ such that for any $i$ and any iteration $k$ the following hold:
\begin{align}
&\text{If } c_i(k)>\left(\sum _{j\in \mathcal{N}_i}w_{ij}\right)\text{, then  } \ \ c_i(k)-\sum _{j\in \mathcal{N}_i}w_{ij}\geq 2\gamma ,\label{ch4:gamma1}\\
&\text{If } \bar{c}_i(k)>\left(\sum _{j\in \mathcal{N}_i}w_{ij}\right)\text{, then } \ \  \bar{c}_i(k)-\sum _{j\in \mathcal{N}_i}w_{ij}\geq 2\gamma ,\label{ch4:gamma2}\\
&\bar{c}_i(k)\geq 2\gamma ,\label{ch4:gamma3}\\
&\frac{1}{2}-\sum _{j\in \mathcal{N}_i}w_{ij}\geq 2\gamma,\label{ch4:weightgamma}
\end{align}
where  $\bar{c}_i(k)=1-c_i(k)$.

{\bf Remark:} \emph{Equations \eqref{ch4:gamma1}-\eqref{ch4:gamma3} do not hold for the simple linear model of \eqref{ch4:stateEq}. For example, consider a linear model that does not reach consensus in finite time, and suppose that $x_{ave} \in \mathbb{Z}$. Then, since $\lim _{k\rightarrow \infty}x_i(k)=x_{ave}$, we have that $c_i(k)$ can be as close to $1$ as desired, and hence we cannot bound $\bar{c}_i(k)$ by a fixed positive value.}

Let $m(k)$ and $M(k)$ be defined as in \eqref{ch4:minmax}. Let us define the following set:
\begin{equation}
S_k=\{\mathbf{y}\in \mathbb{R}^n, |y_i-m(k)-1|\leq \alpha _i\},
\end{equation}
where $\alpha _i=1-w_{ii}+\gamma$. Note that 
\begin{align*}
\alpha _i&=1-w_{ii}+\gamma \\
&=\sum _{j\in \mathcal{N}_i}w_{ij}+\gamma \\
&\leq \frac{1}{2}-\gamma,
\end{align*}
where the last inequality is due to Eq.~\eqref{ch4:weightgamma}, and thus $\alpha _i \in (0,1/2)$.
The set $S_k$ depends on the iteration $k$ because the value $m$ does. Since according to the system \eqref{ch4:quantizedstateeq2}, $m(k)$ cannot decrease and $M(k)$ cannot increase as indicated earlier, then $S_k$ can only belong to one of the $M(0)-m(0)$ possible compact sets at each iteration $k$. Furthermore, if $S_k$ changes to a different compact set due to an increase in $m$, it cannot go back to the old one as $m$ cannot decrease. Additionally, if $\mathbf{x}(k)\in S_k$, then it is an interior point of the set $S_k$ and not on the boundary because suppose $|x_i(k)-m(k)-1|=\alpha _i$, then either $c_i(k)=\alpha _i =\sum _{j\in \mathcal{N}_i}w_{ij}+\gamma$ which contradicts \eqref{ch4:gamma1} or $\bar{c}_i(k)=\alpha _i =\sum _{j\in \mathcal{N}_i}w_{ij}+\gamma$ which contradicts \eqref{ch4:gamma2}.

Let us define the following candidate Lyapunov function:
\begin{align}
V(k)&=d(\mathbf{x}(k),S_k)\nonumber\\
&=\min _{\mathbf{y}\in S_k}||\mathbf{y}-\mathbf{x}(k)||_1\nonumber\\
&=\min _{\mathbf{y}\in S_k}\sum _{i\in \mathcal{V}}|y_i-x_i(k)|
\end{align}
By minimizing along each component of $\mathbf{y}$ independently, we get
$$V(k)=\sum _i{\max \{|x_i(k)-m(k)-1|-\alpha _i,0\}}.$$

\begin{figure}
\begin{center}
\includegraphics[scale=0.75]{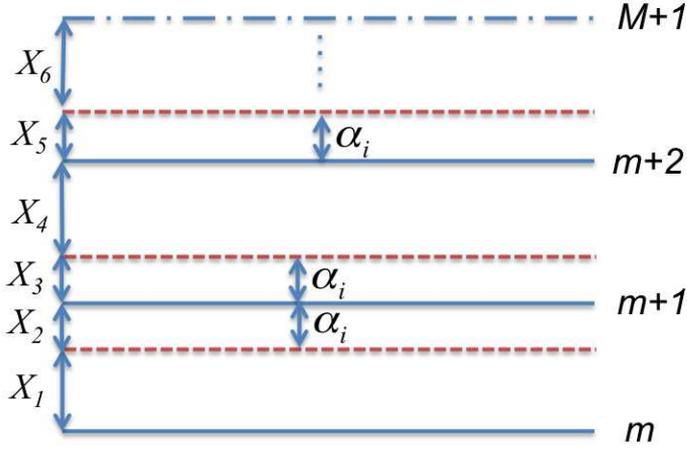}
\caption{Dividing the nodes into sets according to their local values.
} 
\label{ch4:sets}
\end{center}
\end{figure}
Let us determine the change in the proposed candidate Lyapunov function. 
In order to understand the evolution of $\nabla V_k=V(k+1)-V(k)$, we group the nodes  depending on their values at iteration $k$ into 6 sets, $X_1(k)$, $X_2(k)$, $X_3(k)$, $X_4(k)$, $X_5(k)$, and $X_6(k)$ (see Fig.~\ref{ch4:sets}):
\begin{itemize}
\item Node $i\in X_1(k)$ if  \;  $m(k)\leq x_i(k)<m(k)+1-\alpha _i$,
\item Node $i\in X_2(k)$ if \; $m(k)+1-\alpha _i\leq x_i(k)<m(k)+1$,
\item Node $i\in X_3(k)$ if \; $m(k)+1\leq x_i(k)\leq m(k)+1+\alpha _i$,
\item Node $i\in X_4(k)$ if  \; $m(k)+1+\alpha _i< x_i(k)<m(k)+2$,
\item Node $i\in X_5(k)$ if \; $m(k)+2\leq x_i(k)<m(k)+2+\alpha _i$,
\item Node $i\in X_6(k)$ if \; $m(k)+2+\alpha _i\leq x_i(k)$.
\end{itemize}
For simplicity we will drop the index $k$ in the notation of the sets and $m(k)$ when there is no confusion. To have better insights about these sets, we note that if $X_6$ becomes empty at a given iteration, then the set remains empty, i.e., 
\begin{lemma} \label{ch4:lem:empty}
If $X_6(k_0)=\phi$, then $X_6(k)=\phi$ for all $k\geq k_0$.
\end{lemma}
\begin{proof}
If a node $i\notin X_6(k)$, then $\lfloor x_i(k) \rfloor \in \{m, m+1, m+2\}$. So for any node $i$, 
\begin{align*}
x_i(k+1)&=x_i(k)+\sum _{j\in \mathcal{N}_i}w_{ij}L_{ji}\\
&< m+2+\alpha _i,
\end{align*}
where the last equality is due to three possibilities, 
\begin{itemize}
\item if $\lfloor x_i(k)\rfloor=m+2$, then $L_{ji}\leq 0$ for every $j\in \mathcal{N}_i$, and $x_i(k)<m+2+\alpha _i$ since $i\in X_5$ in this case;
\item if $\lfloor x_i(k)\rfloor=m+1$, then $\sum _{j\in \mathcal{N}_i}w_{ij}L_{ji}\leq \sum _{j\in \mathcal{N}_i \cap X_5}w_{ij}\leq \alpha _i$, and $x_i(k)<m+2$  in this case;
\item if $\lfloor x_i(k)\rfloor=m$, then $\sum _{j\in \mathcal{N}_i}w_{ij}L_{ji}\leq \sum _{j\in \mathcal{N}_i}w_{ij}\times 2\leq 2\alpha _i$, and $x_i(k)<m+1$  in this case.
\end{itemize}
Therefore, since $x_i(k+1)< m+2+\alpha _i$, then $i\notin X_6(k+1)$ from the definition of the sets and this ends the proof.
\end{proof}
Note that by a similar reasoning as in Lemma \ref{ch4:lem:empty}, if $\{X_5, X_6\}$ got empty, then it remains empty during all further iterations, and if $\{X_4, X_5, X_6\}$ got empty it remains empty too.

With every iteration, nodes can change their sets. Note that any node can jump in one iteration to a higher set, but the other way around  is not always possible. For example, a node at iteration $k$ in $X_1$ can jump at iteration $k+1$ to $X_6$, but no node outside $X_1$ can get back to it as we will show next.
\begin{lemma}\label{ch4:lem:jump}
If $i\notin X_1(k_0)$, then $i\notin X_1(k)$ for all $k\geq k_0$.
\end{lemma}
\begin{proof}
Let us define $L_i^k$ be the level of node $i$ at iteration $k$, i.e., $L_i^k=\lfloor x_i(k)\rfloor -m(k)$.  Then,
\begin{align*}
x_i(k+1)&=x_i(k)+\sum _{j\in \mathcal{N}_i}w_{ji}L_{ji}\\
&\geq c_i(k)+\lfloor x_i(k)\rfloor +(\sum _{j\in \mathcal{N}_i}w_{ji})(m(k)-\lfloor x_i(k)\rfloor)\\
&= c_i(k)+L_i^k+m(k)+(\sum _{j\in \mathcal{N}_i}w_{ji})(-L_i^k)\\
&=  m(k)+c_i(k)+w_{ii}L_i^k\\
&\geq m(k)+1-\alpha _i ,
\end{align*}
and $i\notin X_1(k+1)$. The last inequality is due to two possibilities, 
\begin{itemize}
\item if $i\in X_2(k)$ then $L_i^k=0$, and  $m(k)+c_i(k)=x_i(k)\geq m(k)+1-\alpha _i$,
\item otherwise $L_i^k\geq 1$, so $m(k)+c_i(k)+w_{ii}L_i^k\geq m(k)+w_{ii}\geq m(k)+1-\alpha _i$.
\end{itemize}
\end{proof}
 Therefore, due to Lemma \ref{ch4:lem:jump} the increase $ V(k)$ is due to nodes changing to a higher set. However, any node changing its set to a higher one, should have neighbors in the higher sets that cause $V(k)$ to decrease by at least the same amount. To make this a formal argument we give the following lemma:
\begin{lemma}\label{ch4:Vnegative}
Consider the quantized system \eqref{ch4:quantizedstateeq2}. Suppose that Assumption~\ref{ch4:ass1} holds. If $m(k+1)=m(k)$, we have $$\nabla V_k \leq 0.$$
\end{lemma}
\begin{proof}
We define $\nabla _iV_k$ as follows:
\begin{align}
\nabla _iV_k&\triangleq \max \{|x_i(k+1)-m-1|-\alpha _i,0\}\nonumber\\
&\hspace*{0.75cm}-\max \{|x_i(k)-m-1|-\alpha _i,0\},
\end{align}
from which it is evident that $\nabla V_k=\sum _{i\in \mathcal{V}}\nabla _iV_k.$
Since only nodes moving from a set $X_s$ to a higher set $X_t$ where $t\geq \max \{s,4\}$ can increase $V(k)$ (we will use the expression $X_s\rightarrow X_t$ to denote the transition of a node that belongs to the set $X_s$ at iteration $k$ to the set $X_t$ at iteration $k+1$), then we can enumerate all the possible transitions of nodes that can cause $V(k)$ to increase:
\begin{enumerate}
\item $X_1(k)\rightarrow X_t(k+1) \; , t\geq 4$,
\begin{align*}
\nabla _iV_k&=\max \{|x_i(k+1)-m-1|-\alpha _i,0\}-\max \{|x_i(k)-m-1|-\alpha _i,0\}\\
&=\left( x_i(k+1)-m-1-\alpha _i\right)-\left( 1+m-x_i(k)-\alpha _i\right)\nonumber\\
&=x_i(k)+\sum _{j\in \mathcal{N}_i}w_{ij}\left(\lfloor x_j(k)\rfloor - \lfloor x_i(k)\rfloor\right)-m-1 - m-1+x_i(k)\nonumber\\
&=\sum _{j\in \mathcal{N}_i}w_{ij}L_{ji}-2(m+1-x_i(k))\nonumber \\
&=\sum _{j\in \mathcal{N}_i}w_{ij}L_{ji}-2\bar{c}_i(k)\nonumber \\
&=\sum _{j\in \mathcal{N}_i}w_{ij}L_{ji}-2(\alpha_i(k)-\alpha _i(k)+\bar{c}_i(k))\nonumber\\
&=(\sum _{j\in \mathcal{N}_i\cap \{X_3,X_4\}}w_{ij})+(\sum _{j\in \mathcal{N}_i\cap X_5}w_{ij}\times 2) +(\sum _{j\in \mathcal{N}_i\cap X_6}w_{ij}L_{ji}) \nonumber\\
&\hspace*{0.3cm}-2(\sum _{j\in \mathcal{N}_i}w_{ij}+\gamma +(\bar{c}_i(k)-\alpha _i))\nonumber\\
&\leq \underbrace{(\sum _{j\in \mathcal{N}_i\cap X_6}w_{ij}L_{ji})}_{\geq 0} -4\gamma .
\end{align*}
\item $X_2(k)\rightarrow X_t(k+1) \; , t\geq 4$, and the change in the Lyapunov function due to these nodes  is as follows:
\begin{align*}
\nabla _iV_k&=\max \{|x_i(k+1)-m-1|-\alpha _i,0\}\nonumber\\
&\hspace*{0.5cm}-\max \{|x_i(k)-m-1|-\alpha _i,0\}\\
&=\left( x_i(k+1)-m-1-\alpha _i\right)-0\nonumber\\
&=x_i(k)+ \sum _{j\in \mathcal{N}_i}w_{ij}L_{ji}-m-1-\alpha _i\nonumber\\
&=\sum _{j\in \mathcal{N}_i}w_{ij}L_{ji}-\alpha _i-\bar{c}_{i}(k)\nonumber\\
&=(\sum _{j\in \mathcal{N}_i\cap \{X_3,X_4\}}w_{ij})+(\sum _{j\in \mathcal{N}_i\cap X_5}w_{ij}\times 2) \nonumber\\
&\hspace*{0.5cm}+(\sum _{j\in \mathcal{N}_i\cap X_6}w_{ij}L_{ji}) -\sum _{j\in \mathcal{N}_i}w_{ij}-\gamma -\bar{c}_i(k)\nonumber\\
&\leq \underbrace{(\sum _{j\in \mathcal{N}_i\cap X_5}w_{ij})}_{\geq 0}+\underbrace{(\sum _{j\in \mathcal{N}_i\cap X_6}w_{ij}L_{ji})}_{\geq 0} -2\gamma .
\end{align*}
\item $X_3(k)\rightarrow X_t(k+1) \; , t\geq 4$, then
\begin{align*}
\nabla _iV_k&=x_i(k)+ \sum _{j\in \mathcal{N}_i}w_{ij}L_{ji}-m-1-\alpha _i\nonumber\\
&=\sum _{j\in \mathcal{N}_i}w_{ij}L_{ji}-(\alpha _i-c_{i}(k))\nonumber\\
&=(\sum _{j\in \mathcal{N}_i\cap \{X_1,X_2\}}w_{ij}\times (-1))+(\sum _{j\in \mathcal{N}_i\cap X_5}w_{ij}) \nonumber\\
&\hspace*{0.5cm}+(\sum _{j\in \mathcal{N}_i\cap X_6}w_{ij}L_{ji}) -(\alpha _i-c_{i}(k))\nonumber\\
&\leq \underbrace{(\sum _{j\in \mathcal{N}_i\cap X_5}w_{ij})}_{\geq 0}+\underbrace{(\sum _{j\in \mathcal{N}_i\cap X_6}w_{ij}L_{ji})}_{\geq 0} -\gamma .
\end{align*}
\item $X_4(k)\rightarrow X_t(k+1) \; , t\geq 4$, then
\begin{align*}
\nabla _iV_k&=\sum _{j\in \mathcal{N}_i}w_{ij}L_{ji}\nonumber\\
&\leq \underbrace{\left(\sum _{j\in \mathcal{N}_i\cap X_5}w_{ij}\right)}_{\geq 0}+\underbrace{\left(\sum _{j\in \mathcal{N}_i\cap X_6}w_{ij}L_{ji}\right)}_{\geq 0}.
\end{align*}
\item $X_5(k)\rightarrow X_t(k+1) \; , t\geq 5$, then 
\begin{align*}
\nabla _iV_k&=\sum _{j\in \mathcal{N}_i}w_{ij}L_{ji}\nonumber\\
&= \underbrace{\left(\sum _{j\in \mathcal{N}_i\cap X_6}w_{ij}L_{ji}\right)}_{\geq 0}+ \underbrace{\left(\sum _{j\in \mathcal{N}_i,j\notin X_6}w_{ij}L_{ji}\right)}_{\leq 0}.
\end{align*}
\item $X_6(k)\rightarrow X_6(k+1)$,  then
\begin{align*}
\nabla _iV_k&=\sum _{j\in \mathcal{N}_i}w_{ij}L_{ji}\nonumber\\
&= \underbrace{\left(\sum _{j\in \mathcal{N}_i\cap \bar{X}_6^i}w_{ij}L_{ji}\right)}_{\geq 0}+ \underbrace{\left(\sum _{j\in \mathcal{N}_i,j\notin \bar{X}_6^i}w_{ij}L_{ji}\right)}_{\leq 0}.
\end{align*}
\end{enumerate}
where the set $\bar{X}_6^i$ is the set of nodes such that $j \in \bar{X}_6^i$ if $x_j(k)\geq x_i(k)$.

Notice that the positive component in $\nabla V_k$ because of a node $s$ belonging to one of the presented $6$ possibilities is only due to a neighbor $p$ in $\{X_5(k), X_6(k)\}$ such that $x_p(k)\geq x_s(k)$. Then $p$ can belong to two possible sets: $X_5$ or $X_6$.

Suppose first that $p\in X_6(k)$,  let $A$ be the increase in $\nabla _sV_k$, then this increase is as follows: 
$$A=w_{ps}L_{ps}>0,$$
but this increase is decreased again in $\nabla _pV_k$ since a node in $X_6(k)$ cannot drop below $X_4(k+1)$, we can write:
\begin{align*}
\nabla _pV_k&=\max \{|x_p(k+1)-m-1|-\alpha _p,0\}\nonumber\\
&\hspace*{0.5cm}-\max \{|x_p(k)-m-1|-\alpha _p,0\}\\
&=\left( x_p(k+1)-m-1-\alpha _p\right)-\left( x_p(k)-1-m-\alpha _p\right)\nonumber\\
&=x_p(k)+\sum _{j\in \mathcal{N}_p}w_{jp}L_{jp}-x_p(k)\nonumber\\
&=\underbrace{w_{sp}L_{sp}}_{-A}+\sum _{j\in \mathcal{N}_p-\{s\}}w_{jp}L_{jp}.
\end{align*} 

Taking the other case, suppose now $p\in X_5$, let $B$ be the increase in $\nabla _sV_k$ of a node $s$ due to its neighbor $p\in X_5$:
$$B=w_{sp}>0,$$
then this increase is decreased again in $\nabla _pV_k$, but we should consider two cases:
\begin{itemize}
\item $p$: $X_5\rightarrow X_m, \; m\geq 4$, then 
\begin{equation}
\nabla _pV_k=\underbrace{w_{ps}L_{sp}}_{\leq -B}+\sum _{j\in \mathcal{N}_p-\{s\}}w_{jp}L_{jp},
\end{equation} 
\item $p$: $X_5\rightarrow X_3$, then 
\begin{align*}
\nabla _pV_k&\leq -1/2\\
&\leq -\sum _{j\in \mathcal{N}_p}w_{pj}\\
&=\underbrace{-w_{ps}}_{-B}-\sum _{j\in \mathcal{N}_p-\{s\}}w_{jp},
\end{align*}
\end{itemize}
and $p$ decreases in the same amount that its neighbor $s$ increased.

{\bf Remark:} \emph{For every positive value that increases $V(k)$, there is a \emph{unique} corresponding negative value that compensates this increase by decreasing $V(k)$. This is because for any link $l\sim (i,j)\in \mathcal{E}$, the increase in $\nabla _iV_k$ due to $l$ forces a decrease  in $\nabla _jV_k$ due to the same link, and so there is one to one mapping between the increased values and the decreased ones.}

 As a result of the discussion we can have the total $\nabla V_k$ cannot increase, namely
 \begin{align*}
 \nabla V_k&=\sum _i\nabla _iV_k\leq 0.
 \end{align*}

\end{proof}

Lemma \ref{ch4:Vnegative} implies that $V(k)$ is non-increasing with time. Now we present two situations under which $V(k)$ is strictly decreasing. The two situations will play an important role in the proof of the main result. 
\begin{itemize}
\item { \bf Situation 1 (S1)} occurs if at iteration $k$ there exists a  link in the network between a node $i\in X_4\cup X_5\cup X_6$ and a node $j\in X_1\cup X_2$, in this case we have, 
\begin{align}
\nabla V_k&\leq -\min \{x_i(k)-m-1-\alpha _i, w_{ij},\bar{c}_j\}\nonumber\\
&\leq -\min \{\gamma ,\delta \},
\end{align}
where $\delta =\min _{(i,j)\in \mathcal{E}}w_{ij}>0.$
\item {\bf Situation 2 (S2)} occurs if at iteration $k$ there exists any link in the network between a node $i\in X_5\cup X_6$ and a node $j\in X_3$, in this case we have, 
\begin{align}
\nabla V_k&\leq -\min \{\alpha _j-c_j(k), w_{ij}\}\nonumber\\
&\leq -\min \{\gamma ,\delta \} .
\end{align}
\end{itemize}
 
\subsection{Proof of Main Result} \label{ch4:subsec:MR}
 
To show that $V(k)$ is eventually decreasing, we have to introduce some more notation. Let $$R(k_0)=\min\{k-k_0; k> k_0, \nabla V_k \leq -\beta\},$$ 
where $\beta >0$ is a positive constant. Notice that if either S1 or S2 occurs at time $T_0>k_0$, then $R(k_0)\leq T_0-k_0$ by considering $\beta=\min \{\gamma ,\delta \}$, i.e., $R(k_0)$ is upper bounded by the minimum time for at least one of the two situations to occur.  We will show that if there exists at least one node in $\{X_4, X_5, X_6\}$ at $k_0$ and $m(k)=m(k_0)$ for $k< R(k_0)+k_0$, then  we can have  a fixed upper bound on $R(k_0)$.  If we looked at the values of the nodes in the network at any iteration $k_0$, we can see that if $k<k_0+R(k_0)$, the network has a special structure: only nodes in $\{X_1,X_2,X_3\}$ have links between each other, nodes in $X_3$ can also have links to $X_4$, but not to $\{X_5,X_6\}$. Nodes in $\{X_5,X_6\}$ can only be connected to $X_4$ (see Fig.~\ref{ch4:situation}).  
\begin{figure}
\begin{center}
\includegraphics[scale=0.5]{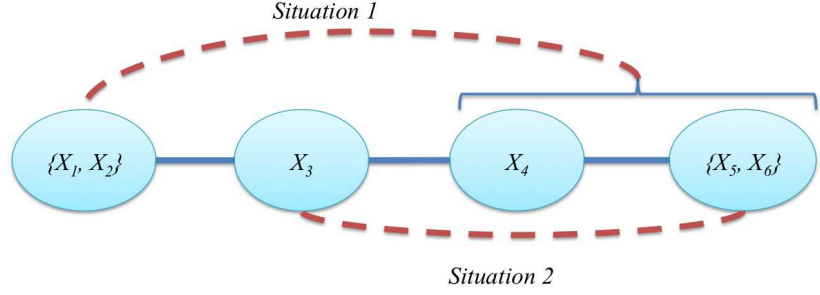}
\caption{The solid lines (blue links) identify the network structure at any iteration $k_0\leq k<k_0+R(k_0)$, while if a dotted link (in red) appears, then $V(k)$ strictly decreases.
} 
\label{ch4:situation}
\end{center}
\end{figure}
Moreover, the values of nodes in $X_3$ cannot increase due to the link between $X_3$ and $X_4$. To see this, let $i\in X_3$ and $s\in X_4$ where $s\in \mathcal{N}_i$. Then we have:
$$x_i(k+1)=x_i(k)+w_ {is}L_{si}+\sum _{j\in \mathcal{N}_i -\{s\}}w_{ij}L_{ji},$$
but since $\lfloor x_i(k) \rfloor=\lfloor x_s(k) \rfloor$, we have $L_{is}=0$ and thus $x_i(k+1)=x_i(k)+\sum _{j\in \mathcal{N}_i -\{s\}}w_{ij}L_{ji}$, so nodes in $X_4$ do not have any effect on nodes in $X_3$ and the values of nodes in $X_3$ cannot increase for all $k<k_0+R(k_0)$ (we will get back to this issue later). 

To find the number of iterations for a dotted (red) link to appear, we define the following function for nodes in $\{X_1, X_2, X_3\}$:
\begin{equation}
f(i,k)=
\begin{cases}
1 &\text{ if }i\in \{X_1(k),X_2(k)\},\\
0 &\text{ if }i\in X_3(k),
\end{cases}
\end{equation}
and let $T_i(k_0,k)$ be the number of times a node $i$ is in $\{X_1, X_2\}$ in the time interval between $k_0$ and $k$, i.e.,
$$T_i(k_0,k)=\sum _{t=k_0}^{t=k}f(i,t).$$

In fact, we can partition the nodes in $\{X_1,X_2,X_3\}$ depending on their distance to nodes in $X_4$. Let $r_i$ be the shortest path distance from a node $i\in \{X_1,X_2,X_3\}$ to the set $X_4$ (i.e., $r_i=\min _{j\in X_4}r_{ij}$ where $r_{ij}$ is the number of hops following the shortest path from $i$ to $j$). We define the set $D_u$ where $u=1,\dots , r$ and $r=\max _ir_i$ as the set of nodes such that $i\in D_u$ if and only if $u=r_i$.   For example, $D_1$ contains nodes that have direct neighbors in $X_4$, $D_2$ contains the nodes that do not have direct neighbors in $X_4$ but there is a node in $X_4$ found 2 hops away, and so on.
Moreover, for any node  $i\in D_u$ such that $u>1$, we can find at least one neighbor $j\in D_{u-1}$. Let $P(i)$ be any one of these neighbors, referred to as the parent of $i$. It is important to note that any node in $D_u$ remains in the set as long as non of the situation has occurred, i.e., the sets $D_u$ for $u=1,\dots ,r$ considered at iteration $k_0$ do not change their elements for $k_0\leq k<k_0+R(k_0)$. We can now obtain the following lemma:

\begin{lemma}\label{ch4:LemmaParent}
If  $\{X_4,X_5,X_6\}\neq \phi$ at an iteration $k_0$, and $m(k)=m(k_0)$ for  $k_0\leq k< k_0+R(k_0)$, then for any integer $N\in \mathbb{N}$: if
$$T_i(k_0,k)\geq N\times\left(\frac{\alpha _{P(i)}}{w_{iP(i)}}+1\right),$$
then 
$$T_{P(i)}(k_0,k)\geq N.$$
\end{lemma}
\begin{proof}
The proof is based on the observation we mentioned earlier. For any node $s\in X_3$, its neighbors in $X_4$ do no have any effect on $x_s(k+1)$ and it cannot have any neighbor in $\{X_5,X_6\}$ otherwise one of the situations (S1 or S2) occurs and contradicts the assumption $k< k_0+R(k_0)$. Therefore, the decrease of the node $s$ from $X_3$ to $X_2$ can only be due to its neighbors in $\{X_1,X_2\}$.  Let $i\in \{X_1,X_2\}$ be a neighbor of node $s$, then 
\begin{align*}
x_s(k+1)&=x_s(k)+\sum _{j\in \mathcal{N}_s}w_{js}L_{js}\\
&=x_s(k)+w_{is}\times(-1)+\sum _{j\in \mathcal{N}_s\cap \{X_1,X_2\}-\{i\}}w_{js}L_{js}\\
&\leq x_s(k)-w_{is}\\
&=1 +m+ c_s(k)-w_{is},
\end{align*}
and the node $s$ can either drop to $X_2$ or stay in $X_3$ depending on the resulting value $x_s(k+1)$. And since $c_s(k)\leq \alpha _s$ and $x_s(k+1)$ cannot increase if $s$ was in $X_3$ at iteration $k$, then we are sure that if $i$ was in $\{X_1,X_2\}$ for more than $\frac{\alpha _s}{w_{is}}$ iterations (i.e., $T_i(k_0,k)\geq \frac{\alpha _s}{w_{is}}+1$), then $s$ has dropped to $X_2$ at least once (i.e., $T_s(k_0,k)\geq 1$). Thus since $P(i)\in \mathcal{N}_i$, we have
\begin{equation}\label{ch4:impliesEq}
T_i(k_0,k)\geq \left(\frac{\alpha _{P(i)}}{w_{iP(i)}}+1\right)\;\;\;\Longrightarrow \;\;\; T_{P(i)}(k_0,k)\geq 1.
\end{equation}

If $T_i(k_0,k_N)\geq N\times\left(\frac{\alpha _{P(i)}}{w_{iP(i)}}+1\right)$, then we can find $N-1$ iterations, $k_1, k_2, \dots, k_{N-1}$, such that $$T_i(k_{v-1},k_v-1)\geq \left(\frac{\alpha _{P(i)}}{w_{iP(i)}}+1\right)\;\;\; \text{ for } v=1,\dots, N.$$  By \eqref{ch4:impliesEq}, we have $T_{P(i)}(k_{v-1},k_{v}-1)\geq 1$. Therefore,
\begin{align*}
T_{P(i)}(k_0,k)&=\sum _{v=1}^{N-1}T_{P(i)}(k_{v-1},k_{v}-1)+T_{P(i)}(k_{N-1},k)\\
&\geq \left(\sum _{v=1}^{N-1}1\right) + 1\\
&\geq N,
\end{align*}
and the lemma is proved.
\end{proof}

Now we show that there is a fixed upper bound on the time for either of the situations to occur,

\begin{lemma}\label{ch4:LemmaUpperbound}
If  $\{X_4,X_5,X_6\}\neq \phi$ at an iteration $k_0$, and $m(k)=m(k_0)$ for  $k\geq k_0$, then 
$$R(k_0)\leq n\left(1+\frac{1}{2\delta}\right)^{n-1},$$
where $\delta =\min _{(i,j)\in \mathcal{E}}w_{ij}$ is a positive constant ($\delta >0$).
\end{lemma}
 \begin{proof}
 Notice first that for any iteration $\bar{k}\geq k_0$, if $T_i(k_0,\bar{k})\geq 1$ where $i\in D_1$, then situation 1 has occurred and $R(k_0)\leq \bar{k}-k_0$. 
 
 Moreover, since  $m(k)=m(k_0)$ for  $k\geq k_0$, then at every iteration $k$ there is at least one node in $\{X_1,X_2\}$, leading to
 $$\sum _{i\in \{X_1,X_2,X_3\}}T_i(k_0,k)\geq k-k_0.$$
 Let $\bar{k}=k_0+n\left(1+\frac{1}{2\delta}\right)^{n-1}$; then we have 
 $$\sum _{i\in \{X_1,X_2,X_3\}}T_i(k_0,\bar{k})\geq n\left(1+\frac{1}{2\delta}\right)^{n-1},$$
 and there must be a node $i\in D_u$ in this sum such that 
 $$ T_i(k_0,\bar{k})\geq \left(1+\frac{1}{2\delta}\right)^{n-1}.$$
Without loss of generality, we can suppose  $\frac{1}{2\delta}\in \mathbb{N}$. So applying Lemma \ref{ch4:LemmaParent}, we can see that
\begin{align*}
T_i(k_0,\bar{k})&\geq \left(1+\frac{1}{2\delta}\right)^{n-1}\\
&\geq \left(1+\frac{\alpha _{P(i)}}{w_{iP(i)}}\right)\times\left(1+\frac{1}{2\delta}\right)^{n-2},\\
&= \left(1+\frac{\alpha _{P(i)}}{w_{iP(i)}}\right)\times N,
\end{align*}
where $N= \left(1+\frac{1}{2\delta}\right)^{n-2}$, which implies 
$$T_j(k_0,\bar{k})\geq \left(1+\frac{1}{2\delta}\right)^{n-2},$$
where $j=P(i)$ and $j\in D_{u-1}$. Doing this recursively ($u-1$ times), we see that there is a node $s\in D_1$ such that,
 $$T_s(k_0,\bar{k})\geq \left(1+\frac{1}{2\delta}\right)^{n-u},$$
 but since $u\leq r\leq  n$, we have $T_s(k_0,\bar{k})\geq 1$ which means situation S1 occurred because $s\in D_1$. Therefore,
 \begin{align*}
 R(k_0)&\leq \bar{k}-k_0\\
 &\leq n\left(1+\frac{1}{2\delta}\right)^{n-1},
 \end{align*}
 and the lemma is proved.
 \end{proof}

We also need the following lemma, 
\begin{lemma}\label{ch4:LemmaBound}
Suppose Assumption~\ref{ch4:ass1} holds. Let $\beta =\min\{\gamma , \delta \} $, then for the quantized system \eqref{ch4:quantizedstateeq2},  at any time $k_0$, there is a finite time  $k_1\geq k_0$  such that for $k\geq k_1$, either   $\{X_4,X_5,X_6\}=\phi$ or $m(k)>m(k_0)$. Moreover,
$$k_1\leq k_0+n\left(\frac{V(k_0)}{\beta}+1 \right)\left( \frac{1}{2\delta}+1\right)^{n-1}.$$
 \end{lemma}
\begin{proof}
Let us prove it by contradiction. Suppose  that $\{X_4,X_5,X_6\}\neq \phi$ and $m(k)=m(k_0)$ for  $k\geq k_0$. Therefore we can apply Lemma \ref{ch4:LemmaUpperbound} to show that there is an upper bound $R(k_0)$ for situations S1 or S2 to occur. Whenever one of the situations occurs, we have $\nabla V_k \leq -\beta$, otherwise $\nabla V_k \leq 0$. For $k>k_0+ n\left(\frac{V(k_0)}{\beta}+1 \right)\left( \frac{1}{2\delta}+1\right)^{n-1}$, we have that situations S1 or S2 have occurred at least $\left(\frac{V(k_0)}{\beta}+1 \right)$ times; then 
\begin{align*}
V(k)&\leq V(k_0)-\beta \times \left(\frac{V(k_0)}{\beta}+1 \right) \leq -\beta <0,
\end{align*}
which is a contradiction since $V(k)\geq 0$ is a Lyapunov function. As a result,  there exists an iteration $k_1$ satisfying  $k_1\leq k_0 + n\left(\frac{V(k_0)}{\beta}+1 \right)\left( \frac{1}{2\delta}+1\right)^{n-1}$ such that for $k\geq k_1$, either   $\{X_4,X_5,X_6\}=\phi$ or $m(k)>m(k_0)$.
\end{proof}

We are now ready to prove the following propositions, 
\begin{prop}\label{ch4:PropMain}
Consider the quantized system \eqref{ch4:quantizedstateeq2}. Suppose that Assumption~\ref{ch4:ass1} holds. Then for any initial value $\mathbf{x}(0)$, there is a finite time iteration where $\{X_4,X_5,X_6\}=\phi$. 
\end{prop}
\begin{proof}
The value $m(k)$ cannot increase more than $M(0)-m(0)$ number of times because $M(k)$ is non-increasing. Therefore, applying Lemma \ref{ch4:LemmaBound} for  $M(0)-m(0)$ times, we see that  $\{X_4,X_5,X_6\}=\phi$ in a finite number of iterations. 
\end{proof}

Proposition \ref{ch4:PropMain} shows that in fact the nodes are restricted in a finite number of iterations to the  sets $\{X_1,X_2,X_3\}$. In fact, we can even show a  stronger result, that either $X_1$ or $X_3$ can be nonempty, but not both. This is given in the next proposition.

\begin{prop}\label{ch4:PropMain2}
Consider the quantized system \eqref{ch4:quantizedstateeq2}. Suppose that Assumption~\ref{ch4:ass1} holds. Then for any initial value $\mathbf{x}(0)$, there is a finite time iteration where either $\{X_3,X_4,X_5,X_6\}=\phi$ or $\{X_1,X_4,X_5,X_6\}=\phi$. 
\end{prop}
\begin{proof}
Due to Proposition \ref{ch4:PropMain}, we can find a finite time $T$ such that $\{X_4,X_5,X_6\}=\phi$. Without loss of generality, we consider $T=0$. 
In fact, a third situation that can strictly decrease $V(k)$ occurs when there is a link between a node in $X_1$ and a node in $X_3$.  Fig.~\ref{ch4:situation3} shows the network structure.
\begin{figure}
\begin{center}
\includegraphics[scale=0.5]{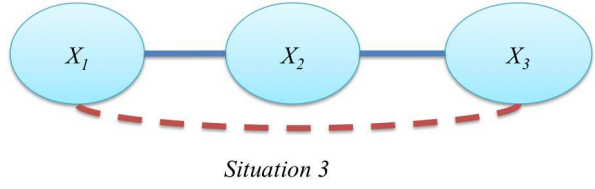}
\caption{The solid lines (blue links) identify the network structure at any iteration $k_0\leq k<k_0+R(k_0)$, while if the dotted link (in red) appears, then $V(k)$ strictly decreases.
} 
\label{ch4:situation3}
\end{center}
\end{figure}
If Situation 3 (S3) occurs and $(ij)\in \mathcal{E}$ where $i\in X_1$ and $j\in X_3$, then 
\begin{align}
\nabla V_k&\leq -\min \{\bar{c}_i(k)-\alpha _i, w_{ij}\}\nonumber\\
&\leq -\min \{\gamma ,\delta \}.
\end{align}
In fact, similar to the reasoning along this subsection, we can bound the number of iterations for S3 to occur.  The bound is exactly the same as the one developed for the other situations. Instead of repeating the derivations, the proof reads roughly the same starting from the beginning of Subsection \ref{ch4:subsec:MR} but by replacing $X_1$, $X_2$, and $X_3$ by $\phi$, replacing $X_2$ by $X_3$, replacing $X_3$ by $X_2$, replacing $X_4$ by $X_1$, and finally replacing the condition $m(k)=m(k_0)$ by $X_3 \neq \phi$. Thus, Lemma \ref{ch4:LemmaBound} will read as follows: Suppose Assumption~\ref{ch4:ass1} holds. Let $\beta =\min\{\gamma , \delta \} $, then for the quantized system \eqref{ch4:quantizedstateeq2},  at any time $k_0$, there is a finite time  $k_1\geq k_0$  such that for $k\geq k_1$, either   $X_1=\phi$ or $X_3=\phi$. This ends the proof.
\end{proof}

\begin{prop}\label{ch4:PropMain3}
Consider the quantized system \eqref{ch4:quantizedstateeq2}. Suppose that Assumption~\ref{ch4:ass1} holds and let $\alpha =\max_i\alpha _i$. Then for any initial value $\mathbf{x}(0)$, there is a finite time iteration where either
\begin{itemize}
\item the values of nodes are cycling in a small neighborhood around the average such that :
\begin{equation}\begin{cases}
|x_i(k)-x_j(k)|\leq \alpha _i+\alpha _j \text{ for all } i,j \in \mathcal{V}\\
|x_i(k)-x_{ave}|\leq 2\alpha \text{ for all } i\in \mathcal{V},
\end{cases}\end{equation}
\item or the quantized values have reached consensus, i.e.,
\begin{equation}
\begin{cases}
\lfloor x_i(k)\rfloor=\lfloor x_j(k)\rfloor \text{ for all } i,j \in \mathcal{V}\\
|x_i(k)-x_{ave}|< 1 \text{ for all } i\in \mathcal{V}.
\end{cases}
\end{equation}
\end{itemize}
\end{prop}

\begin{proof}
The two possibilities are consequence of the two possible cases of Proposition \ref{ch4:PropMain2}, 
\begin{itemize}
\item Case $\{X_1,X_4,X_5,X_6\}=\phi$. Then all nodes are in $\{X_2,X_3\}$ and by the definition of the sets we have $|x_i(k)-x_j(k)|\leq \alpha _i+\alpha _j \text{ for all } i,j \in \mathcal{V}$, so nodes are cycling (due to Proposition \ref{ch4:PropCycle}) around $m+1$. Moreover, since the average is conserved from Eq.~\eqref{ch4:AveCons}, we have:
\begin{align*}
|x_i(k)-x_{ave}|&=|x_i(k)-x_{ave}(k)|\\
&\leq |\max _ix_i(k)-\min _ix_i(k)|\\
&\leq 2\max_i\alpha _i\\
&=2\alpha,
\end{align*}
\item Case $\{X_3,X_4,X_5,X_6\}=\phi$. Then all nodes are in $\{X_1,X_2\}$ and by the definition of the sets we have reached quantized consensus. Since for any $i$ and $j$ we have $c_i(k),c_j(k)\in [0,1)$, then $|x_i(k)-x_j(k)|<1$ and as in the above due to Eq.~\eqref{ch4:AveCons}, we have $|x_i(k)-x_{ave}|< 1$.
\end{itemize}
\end{proof}

\section{Discussion}\label{ch4:D}

Propositions \ref{ch4:PropCycle}  shows that the uniform quantization on communications given by the model of this report can have a very important cyclic property. Up to our knowledge, this is the first work in deterministic quantized algorithms that shows this cyclic effect of nodes' values and it is also shown by Proposition \ref{ch4:PropMain3} that the cyclic values can be control by a simple distributed adjustment of the weights. This can have an important impact on the design of quantized communication algorithms.\footnote{Pattern generation (as for cyclic systems) plays an important role in the design of many mechanical and electrical systems \cite{Brockett:1997Cyc}. } For example, due to the cyclic effect,  nodes can use the history of their values to reach asymptotic convergence as the following proposition shows:
\begin{col}\label{ch4:PropMainnew}
Consider the quantized system \eqref{ch4:quantizedstateeq2}. Suppose that Assumption~\ref{ch4:ass1} holds. Then for any initial value $\mathbf{x}(0)$, if $y_i(k)$ is an estimate of the average at node $i$ following the recursion:
\begin{equation}
y_i(k)=\frac{k}{k+1}y_i(k-1)+\frac{1}{k+1}x_i(k), \; \; \forall i\in \mathcal{V}, 
\end{equation}
where $y_i(0)=x_i(0)$, then $y_i(k)$ is converging,
\begin{equation}
\lim _{k\rightarrow \infty} y_i(k)=y_i^*, \; \; \forall i\in \mathcal{V}, 
\end{equation}
having 
$$|y_i^*-x_{ave}|\leq 1.$$
\end{col}
\begin{proof}
The state equation of $y_i(k)$ for a node $i$ is give by 
\begin{align*}
y_i(k)&=\frac{k}{k+1}y_i(k-1)+\frac{1}{k+1}x_i(k)=\frac{1}{k+1}\sum _{t=0}^{t=k}x_i(t)\\
&=\frac{1}{k+1}\left(\sum _{t=0}^{t=T_{conv}-1}x_i(t)\right)+\frac{1}{k+1}\left(\sum _{t=T_{conv}}^{t=k}x_i(t)\right),
\end{align*}
where $T_{conv}$ is the finite time iteration when the nodes' values start cycling. As $k$ approaches infinity, the left part in the sum vanishes while the right part converges to the average of the values in a cycle, i.e.
$$\lim _{k\rightarrow \infty}y_i(k)=y_i^*=\frac{1}{P}\sum _{t=T_{conv}}^{t=T_{conv}+P-1}{x_i(t)},$$
where $P$ is the cycle period. 
Since for $k\geq T_{conv}$ we have $|x_i(k)-x_{ave}|\leq 1$ from Proposition \ref{ch4:PropMain3}, then $|y_i^*-x_{ave}|\leq 1$.
\end{proof}

Moreover, since the final behavior of the system depends on the initial values as shown by Proposition \ref{ch4:PropMain3}, we give here a  condition on the initial values for the nodes to reach quantized consensus in networks: 
\begin{col}\label{ch4:PropFinal}
Consider the quantized system \eqref{ch4:quantizedstateeq2}. Suppose that Assumption~\ref{ch4:ass1} holds. If the initial values $\mathbf{x}(0)$ satisfy,
\begin{equation}\label{ch4:CondInit}
 \alpha \leq x_{ave}-\lfloor x_{ave} \rfloor \leq  1-\alpha,
\end{equation}
then the network reaches quantized consensus.
\end{col}
\begin{proof}
If the system was cyclic, then for any node $i\in \mathcal{V}$, we have $i\in \{X_1, X_2\}$, so $x_i(k)\in [m+1-\alpha _i, m+1 +\alpha _i]$. This implies that $x_{ave}(k)\in [m+1-\alpha _i, m+1 +\alpha _i]$, but since the average is conserved (from equation \eqref{ch4:AveCons}), it also implies that $x_{ave}\in [m+1-\alpha _i, m+1 +\alpha _i]$. From the latter condition, we see that if $\alpha < x_{ave}-\lfloor x_{ave} \rfloor < 1-\alpha$,  the system cannot be cyclic, and by Proposition \ref{ch4:PropMain3}, it must reach quantized consensus.
\end{proof}
\subsection{Design of weights with arbitrarily small error}

If the system has reached quantized consensus, the values of the agents' agreement variables become stationary and the deviation of these values from the average is no larger than 1. In the case when the system does not reach quantized consensus but becomes cyclic, Proposition \ref{ch4:PropMain3} shows that the deviation of nodes' values from the average is upper bounded by $2\alpha$ where $\alpha =\max_i\alpha _i$. Moreover the deviation can be made arbitrarily small by adjusting the weights in a distributed manner. Toward that end, we propose the following modified Metropolis weights:
\begin{eqnarray*}
w_{ij} &=& \frac{1}{C\left(\text{max}\{d_i,d_j\}+1\right)},  \ \ \forall (i,j)\in \mathcal{E}\\
w_{ii} &=& 1-\sum _{j\in \mathcal{N}_i}w_{ij}, \ \ \forall i\in \mathcal{V}
\end{eqnarray*}
where $C$ is any rational constant such that $C\geq 2$. It can be easily checked that the proposed weights satisfy Assumption~\ref{ch4:ass1}. Moreover, in addition to its distributed nature, the choice of $C$ can be used to define the error. Notice that  for any $i\in \mathcal{V}$, we have $w_{ii}>1-\frac{1}{C}\geq 1-\frac{1}{C}+\gamma$, so
\begin{align*}
\alpha \leq \frac{1}{C},
\end{align*} 
which shows that  given an arbitrary level of precision known to all the agents, the agents can choose the weights with  large enough $C$ in a distributed manner, so that the neighborhood of the cycle  will be close to the average with the given precision. Notice that if $x_{ave}\neq \lfloor x_{ave} \rfloor$, then for $\alpha$ small enough, the system cannot be cyclic and only quantized consensus can be reached (Corollary~\ref{ch4:PropFinal}). In other words, for systems starting with different initial values, having a smaller $\alpha$ leads more of these systems to converge to quantized consensus (and of course if they cycled, they will cycle in a smaller neighborhood as well due to  Proposition~\ref{ch4:PropMain3}).

It is worth mentioning that this arbitrarily small neighborhood weight design has a trade-off with the speed of convergence of quantized consensus protocol  (small error weight design leads to slower convergence).

\section{Simulations}\label{ch4:sec:S}

\begin{figure}
\begin{center}
\includegraphics[scale=0.45]{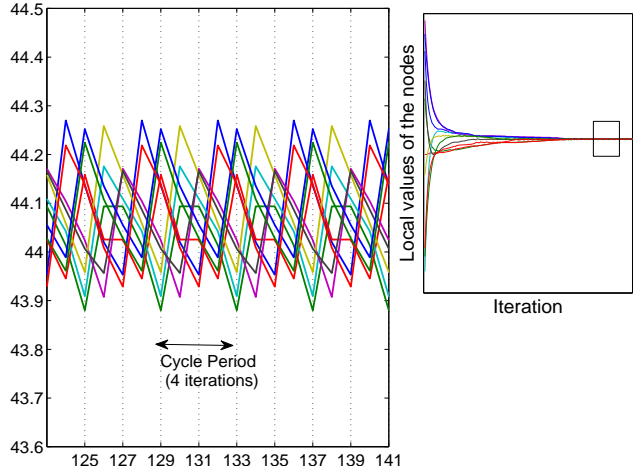}
\caption{The nodes' values are entering into a cycle.
} 
\label{cycle}
\end{center}
\end{figure}

In this section, we present some simulations to demonstrate the theoretical results in the previous section. The weights for the simulations satisfy Assumption 1 and are the modified Metropolis weights with $C=2$, i.e. $$w_{ij} = \frac{1}{2\left(\text{max}\{d_i,d_j\}+1\right)}  \ \ \forall (i,j)\in \mathcal{E}.$$

\begin{figure}
\begin{center}
\includegraphics[scale=0.45]{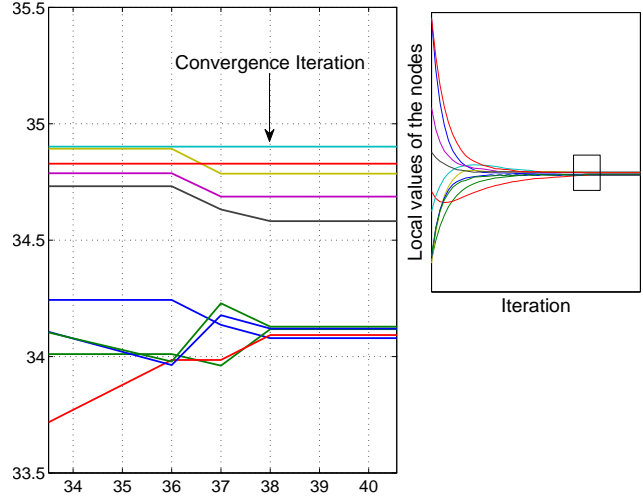}
\caption{The nodes' values are converging.
} 
\label{consensus}
\end{center}
\end{figure}

\subsection{A Simple Network}
Proposition \ref{ch4:PropMain3} shows that depending on the initial state $\mathbf{x}(0)$, the system reaches in finite time one of the two possibilities: 1) cyclic, 2)quantized consensus. We show on a network of $10$ nodes with initial values  selected uniformly at random from the interval $[0,100]$  that both of these are possible. Fig.~\ref{cycle}, shows that after a certain iteration, the nodes' values enter into a cycle of period 4 iterations, while Fig.~\ref{consensus} shows that starting from different initial values, all the 10 nodes reach quantized consensus in finite time. Mainly, at iteration 38, all nodes' values are between 34 and 35; therefore, we have $$\lfloor x_i(k) \rfloor=34 \hspace*{1cm} \forall i=1,\dots ,10, \;\; \forall k\geq 38.$$

\subsection{Random Graphs}

To further simulate our theoretical results, we need to select some network model. The simulations are done on random graphs: Erd\"os-Renyi (ER) graphs and Random Geometric Graphs (RGG), given that they are connected. The random graphs are generated as follows: 
\begin{itemize}
\item For the ER random graphs, we start from $n$ nodes fully connected graph, and then every link is removed from the graph by a probability $1-P$ and is left there with a probability $P$. We have tested the performance for different probabilities $P$ given that the graph is connected.
\item For the RGG random graphs, $n$ nodes are thrown uniformly at random on a unit square area, and any two nodes within a connectivity radius $R$ are connected by a link (the connectivity radius $R$ is selected as $R=\sqrt{c\times \frac{\log (n)}{n}}$ where $c$ is a constant that is studied by wide literature on  RGG for connectivity). We have tested the performance for different connectivity radii given that the graph is connected. It is known that for a small connectivity radius, the nodes tend to form clusters. 
\end{itemize} 

Since Proposition \ref{ch4:PropMain3} shows that the system would reach one of the cases in finite time, let us define $T_{conv}$ be this time. Notice that if nodes enter the  cyclic states (case 1), the Lyapunov function is null because for all $i\in \mathcal{V}$ and $k\geq T_{conv}$, we have $x_i(k)\in [m+1-\alpha_i,m+1+\alpha_i]$ , so we can write, 
$$V(k)=0 \; \; \forall k\geq T_{conv}.$$
However, if nodes reached quantized convergence (case 2), then the Lyapunov function is a constant because for all $i\in \mathcal{V}$ and $k\geq T_{conv}$, we have $x_i(k)\in [m,m+1]$, so we can write,
$$V(k)=cte \; \; \forall k\geq T_{conv}.$$

\subsubsection{Lyapunov Function}

\begin{figure}
\begin{center}
\includegraphics[scale=0.45]{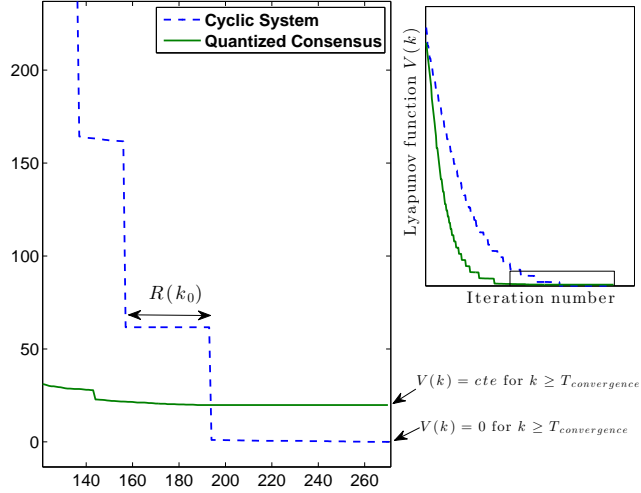}
\caption{The system Lyapunov function $V(k)$.
} 
\label{lyapunov}
\end{center}
\end{figure}

Fig.~\ref{lyapunov} shows the Lyapunov functions for the two different cases on an RGG with $100$ nodes and $R=0.2146$, where each case corresponds to initial values of nodes selected uniformly at random from the interval $[0,100]$. The figure also shows $R(k_0)$ which is the number of iterations after $k_0$ up till $V(k)$ decreases (S1 or S2 occurs).

\subsubsection{Quantized Consensus}

Given that we are considering Metropolis weights with $C=2$, then the system satisfies \eqref{ch4:CondInit} if initial states are such that $x_{ave}-\lfloor x_{ave}\rfloor =0.5$. We considered $RGG$ and $ER$ graphs of 100 nodes, where the initial condition is chosen as follows: the first $99$ nodes are given uniformly random initial values from the interval $[0,100]$, while the last node is given an initial value such that $x_{ave}-\lfloor x_{ave} \rfloor=0.5$ is satisfied. Therefore,  with these initial values, by applying Corollary \ref{ch4:PropFinal}, the system reaches quantized consensus in finite time $T_{conv}$. Table I shows the mean value over 100 runs of the $T_{conv}$ for the RGG with different connectivity radii, $R_1<R_2<R_3<R_4<R_5$, where $R\in \{0.1357, 0.1517, 0.1858, 0.2146, 0.3717\}$. The results show that the more the graph is connected, the faster the convergence. These results are also shown to be true on ER graphs. Table II shows the mean value over 100 runs of the $T_{conv}$ for the ER with different probability $P$ , $P_1<P_2<P_3<P_4$, where $P\in \{0.04, 0.06, 0.08, 0.10\}$.

\begin{table}
\centering  
\begin{tabular}{|c |c c c c c|} 
\hline                        
&\multicolumn{5}{|c|}{RGG $n=100$}\\
 &  $R_1$ &  $R_2$  &  $R_3$ &  $R_4$  & $R_5$ \\  
\hline                  
$T_{conv}$ & 1965.3 & 1068.9 & 364.3  & 233.3 & 55.9 \\
\hline 
\end{tabular}
\label{tab:first} 
\caption{Convergence time for Random Geometric Graphs (RGG) with different connectivity radii (averaged over 100 runs).} 
\end{table}

\begin{table}
\centering  
\begin{tabular}{|c |c c c c |} 
\hline                        
&\multicolumn{4}{|c|}{ER $n=100$}\\
 &  $P_1=0.04$ &  $P_2=0.06$  &  $P_3=0.08$ &  $P_4=0.10$   \\  
\hline                  
$T_{conv}$ & 161.49 & 99.38 & 66.58 & 43.43  \\
\hline 
\end{tabular}
\label{tab:second} 
\caption{Convergence time for Erdos Renyi (ER) with different probabilities of link existence (averaged over 100 runs).} 
\end{table}

\section{Conclusion}\label{ch4:sec:C}
In this paper, we studied the performance of deterministic distributed averaging protocols subject to communication quantization. We have shown that quantization due to links can force quantization on the state. Depending on initial conditions, the system converges in finite time to either a quantized consensus, or the nodes' values are entering into a cyclic behavior oscillating around the average.

Since the quantized consensus can be considered as a cyclic state with cycle period equal to zero, we will be investigating in future work the cycle period of the system. Moreover, we have just considered  in this paper fixed networks with synchronous iterations, but since the weights for the quantized distributed averaging are selected in a totally distributed way, we are planning on extending this study to include asynchronous updates on time varying networks.

\section{Acknowledgment }
This research was partially supported by the U.S. Air Force Office of Scientific Research (AFOSR) MURI grant \text{FA9550-10-1-0573}. The authors would like to thank Paolo Frasca for the useful discussion of the cyclic example in Section \ref{subsec:CE}.

\bibliographystyle{abbrv}
\bibliography{./sigmetrics_bib,./ji}

\end{document}